\documentclass[11pt,leqno]{article}
\usepackage{amssymb}
\usepackage[reqno]{amsmath}
\usepackage{amsthm}
\usepackage{amsfonts}
\usepackage{comment}
\usepackage{graphicx}
\usepackage{xcolor}
\usepackage{mathtools}
\mathtoolsset{showonlyrefs}
\usepackage{csquotes}
\usepackage{hyperref}
\usepackage{enumitem}
\usepackage{setspace} 
\usepackage{color}
\definecolor{gr}{rgb}{0,0.50, 0.50}
\definecolor{mg}{rgb}{0.85,0,0.85}

\usepackage{pgfplots}
	\usetikzlibrary{pgfplots.fillbetween,arrows}
	\tikzset{dimens1/.style={->,>=stealth}}
	\definecolor{blue1}{HTML}{f4eb0a}
	\definecolor{blue1a}{HTML}{2bc4b0}
	
\newcommand\blfootnote[1]{%
  \begingroup
  \renewcommand\thefootnote{}\footnote{#1}%
  \addtocounter{footnote}{-1}%
  \endgroup
}

\newcommand{\bk}{\color{black}}

\def\R{\mathbb{R}}

\def\Z{\mathbb{Z}}

\def\ds{\displaystyle}
 
\def\e{\varepsilon}
\def\o{\Omega}
\def\ue{u^\e}
\def\uoe{u^\varepsilon_1} 
\def\ute{u^\varepsilon_2}

\def\lto{L^2(\Omega)}

\def\gae{\Gamma^\varepsilon} 
\def\Ge{\Gamma^\e}

\def\ep{\varepsilon}
\def\te{\mathcal{T}_\varepsilon}

\def\toe{\mathcal{T}^\varepsilon_1}
\def\tte{\mathcal{T}^\varepsilon_2}

\def\rw{\rightarrow}
\def\ru{\rightharpoonup}
\def\G{\Gamma}
\def\n{\nabla}
\def\g{\gamma}
\def\z{\zeta}
\def\te{\mathcal{T}^\e}

\def\beq{\begin{equation}}
\def\eeq{\end{equation}}
\def\ba{\begin{array}}
\def\ea{\end{array}}
\def\d{\mathcal{D}}
\def\f{\varphi}

\newtheorem{theorem}{Theorem}[section]

\newtheorem{lemma}[theorem]{Lemma}
\newtheorem{proposition}[theorem]{Proposition}
\newtheorem{definition}[theorem]{Definition}

\theoremstyle{definition}
\newtheorem{remark}[theorem]{Remark}

\numberwithin{equation}{section}

	\title{On the homogenization of a Signorini-type problem\\ in a domain with inclusions}
	
	\author{Sara Monsurr\`{o}$^{*\dagger}$, Carmen Perugia$^{\ddagger,\dagger}$ and Federica Raimondi$^{*\dagger}$}
	
	\date{}

\begin{document}

	%	\numberwithin{equation}{section}
	\maketitle
	\vskip -8mm
	{\scriptsize{* Universit\`a degli Studi di Salerno, Dipartimento di Matematica, Via Giovanni Paolo II 132, 84084 Fisciano (SA), Italy.\hskip 2mm Email: smonsurro@unisa.it,$\;$ fraimondi@unisa.it\\
\indent $\ddagger$ Universit\`a degli Studi del Sannio, Dipartimento di Scienze e Tecnologie, Via dei Mulini 74, 82100 Benevento, Italy.\hskip 2mm Email: cperugia@unisannio.it
\\ $\dagger$ Member of the Gruppo Nazionale per l'Analisi Matematica, la Probabilit\`a e le loro
Applicazioni (GNAMPA) of the Istituto Nazionale di Alta Matematica (INdAM).}}

	\begin{abstract} 
	In this paper we investigate the effect of a Signorini-type interface condition on the asymptotic behaviour, as $\e$ tends to zero, of problems posed in $\e$-periodic domains with inclusions. The Signorini-type condition is expressed in terms of two complementary equalities involving the jump of the solution on the interface and its conormal derivative via a parameter $\gamma$. Our problem models the heat exchange in  a medium hosting an $\e$-periodic array of thermal conductors in presence of impurities distributed on some regions of the interface. Different limit problems are obtained according to different values of $\g$. 
\blfootnote{\textbf{Keywords: }{Homogenization, Domains with inclusions, Imperfect interface, Signorini-type condition, Variational inequality, Unfolding method.}\vskip 0.3mm}
\blfootnote{\textbf{MSC:} 35B27, 35J25, 35J20, 35J60, 35R35. }
	\end{abstract}
	
\section{Introduction}\label{intro}
The Signorini problem was presented for the first time in \cite{Signo} in the context of linear elasticity, and, eventually, it was given a rigorous mathematical justification in \cite{fichera}. It concerns the equilibrium configuration of an elastic body on a rigid frictionless surface, subject only to its mass forces.\\
Signorini-type boundary conditions appear in many problems in applied mathematics deriving from engineering and physics. We can think, for instance, to lubrication and filtration processes,  hydrodynamics, plasticity, crack theory, optimal control problems, etc.\\
The weak formulations of such kind of problems involve variational inequalities corresponding to nonlinear free boundary-value problems. For mathematical and historical backgrounds we refer to Duvaut and Lions \cite{DL}, Kinderleherer and Stampacchia \cite{KS} and Brezis \cite{Bre}. The unique solvability of these inequalities has been investigated in \cite{LS} and \cite{Signo1} (see also \cite{lions}).\\ 
Successively, this type of problems have been studied in different geometrical settings. 
For instance, the homogenization of boundary value problems involving Signorini boundary conditions has been treated in \cite{GM1}, \cite{GM2} and \cite{MNW} in domains with oscillating boundaries in both linear and nonlinear cases. In \cite{CET}, \cite{CMT} and \cite{pastu} perforated domains have been considered.
Cracks and layers were taken into account in the works of \cite{DGO} and \cite{GMO}. 

In our paper we consider an open bounded domain $\Omega$ of $\mathbb{R}^N$ made up of two disjoint components $\o^\e_1$ and $\o^\e_2$ separated by a Lipschitz-continuous interface $\Gamma^\e$, where $\e$ is a small positive parameter tending to zero. In this domain we deal with the following problem:
\begin{equation}\label{probintro}
\left\{\begin{array}{lllll}
\displaystyle
-\operatorname{div}(A^\e\n \uoe)=f & \hbox{ in } \o_1^\e, \\[2mm]
\displaystyle
-\operatorname{div}(A^\e\n \ute)=f & \hbox{ in } \o_2^\e, \\[2mm]
\displaystyle
(A^\e\n \uoe)\cdot\nu_1^\e=-(A^\e\n \ute)\cdot\nu_2^\e & \hbox{ on } \Ge,\\[2mm]
\displaystyle
[\ue]\geq 0, \quad (A^\e\n \uoe)\cdot\nu_1^\e+\e^\g h^\e[\ue] \geq 0 & \hbox{ on } \Ge,\\[2mm]
\displaystyle
[\ue]((A^\e\n \uoe)\cdot\nu_1^\e+\e^\g h^\e[\ue])=0 & \hbox{ on } \Ge,\\[2mm]
\displaystyle
\uoe =0 & \hbox{ on } \partial \Omega,
\end{array}\right.
\end{equation}
where $\nu_i^\e$ is the unit external normal vector to $\o_i^\e$, $i=1,2$ and $[\cdot]$ is the jump of the solution on the interface. In \eqref{probintro}, the coefficients $A^\e$ and $h^\e$ satisfy suitable hypotheses of uniform ellipticity, boundedness and periodicity, $\g\in \R$ and $f\in L^2(\o)$.\\
Problem \eqref{probintro} can modelize the heat exchange in  a medium hosting an $\e$-periodic array of thermal conductors. Actually, in nature the presence of  impurities distributed in some regions of the interface may give rise to a discontinuity of the temperature field. This is why the Signorini interface condition in the fourth and fifth lines of \eqref{probintro} seems more realistic to describe such real phenomena. Indeed, it means that one can distinguish two a priori unknown subsets of $\gae$ where the couple $u^\e=(\uoe, \ute)$ satisfies the two alternative equalities:
$$[\ue]=0\qquad\text{or}\qquad(A^\e\n \uoe)\cdot\nu_1^\e+\e^\g h^\e[\ue]=0.$$
From a physical point of view there exist a zone of $\gae$ of perfect conduction i.e. absence of jump, and one characterized by an imperfect contact transmission condition. Here, the external thermal flux is proportional to the jump through a coefficient depending on the period $\e$ and on the parameter $\gamma$ representing the order of magnitude of the resistance with respect to $\e$.

Our aim is to investigate the effect of the Signorini interface condition on the asymptotic behaviour, as $\e$ tends to zero, of problem \eqref{probintro}.

We prove that the thermal properties of the limit problem vary according to the different values of the parameter $\gamma$. 
For $\gamma<-1$ and $-1<\gamma<1$  the homogenized problems are linear and standard (see \cite{DoLNTa}, \cite{DoMo}, \cite{Mo}) meaning that the  Signorini-type interface condition in the $\e$-problem does not affect the limit behaviour.

 More precisely, if $\gamma<-1$, we get the classical homogenized problem of Bensoussan-Lions-Papanicolau in a fixed domain, see \cite{ben}.
Actually, the limit problem behaves as in presence, at $\e$-level, of the only transmission condition on the interface of type $[\ue]=0$. Hence, the part of $\Ge$ where the two temperature fields differ is such small that it does not affect the asymptotic behaviour and we do not have a variational inequality anymore.

On the contrary, when $\gamma\in ]-1,1[$, we obtain the same limit behaviour of \cite{DoLNTa} and \cite{DoMo} where, at $\e$-level, the only condition $(A^\e\n \uoe)\cdot\nu_1^\e +\e^\g h^\e[\ue]=0$ is assumed on the interface. More in particular, we get the classical homogenized problem of Cioranescu-S.J. Paulin in the case of a perforated domain with Neumann condition on the boundary of the holes, see \cite{cp}. 
Indeed, the flux related to $u^\e_2$ asymptotically vanishes and the effective diffusion of the first phase is the one obtained when there is no material occupying $\Omega^\e_2$. However, in the limit problem one has $f$ and not $\theta_1 f$ ($\theta_1$ being the proportion of the material occupying $\o_1^\e$) meaning that the heat exchange takes into account also the source term inside the inclusions.

The most interesting cases are $\gamma=-1$ and $\gamma=1$ since the two regions of $\Ge$ where the Signorini conditions are satisfied at $\e$-level are both significant at the limit.

For $\gamma=-1$, we obtain a nonlinear homogenized problem (cf. \eqref{probhom}) where the effective diffusion term takes into account the presence of the contact barrier. Indeed, the associated nonlinear operator is defined through the solution of a cell problem consisting in a variational inequality with a jump on the interface $\G$ (cf. \eqref{ahom}-\eqref{celltircase}). Here, the nonnegativity constraint of the jump on $\gae$ gives rise to a nonnegative jump on the interface $\G$ in the cell problem.

At last, when $\gamma=1$, the Signorini-type interface condition at $\e$-level leads to an obstacle-type homogenized problem (cf. \eqref{obstacle}). Unlike previous cases, at the limit we get two different temperature fields. In some sense, as in the case $\gamma\in ]-1,1[$, the part of the interface where the condition $(A^\e\n \uoe)\cdot\nu_1^\e +\e^\g h^\e[\ue]=0$ is assumed prevails on the one where $[\ue]=0$ holds. However, the gap between the two temperatures on $\gae$ is so big that, 
even if the flux related to $u^\e_2$ asymptotically vanishes as for $\gamma\in ]-1,1[$, the material occupying the inclusions contributes to the heat exchange. Indeed, the initial jump between the two temperature fields does not vanish but, at the limit, it reflects on the whole domain. As a consequence,  the homogenized problem consists in a variational inequality with a lower order term depending on the physical properties of the interface and the gap between the two temperature fields.

The main difficulty of this paper consists in the choice of appropriate test functions allowing to identify the different limit problems.
Indeed, only for $\g<-1$ it is possible to use quite standard arguments.
For $\g=-1$, the required test functions lead to handle products of two weakly convergent sequences. Lower semi-continuity arguments, suitable in the framework of variational inequalities, allow us to manage this issue.
Finally, in the cases $\g\in]-1,1[$ and $\g=1$ we need to consider more regular test functions together with opportune density results. Moreover, when $\g=1$, we take advantage of lower semi-continuity arguments again.

The homogenization of boundary value problems involving this kind of composite domains was performed at first by using the method of Tartar of oscillating test functions in \cite{Mo} and \cite{Moer} for $\gamma \leq-1$ and in \cite{DoMo} for $\gamma >-1$. Eventually, the same results were obtained by means of the unfolding method in \cite{DoLNTa} and \cite{donngu} in the linear and nonlinear case, respectively.
Further problems involving two-component domains with jump conditions have been treated in \cite{AmAnT}, \cite{AMR}, \cite{RCk}, \cite{RT}, \cite{FMPhomo}, \cite{mpr}, \cite{Rai} (and the references therein) for what concerns the asymptotic behaviour, and in \cite{FMP}, \cite{MNP}, \cite{MNP2} and \cite{MP1} as for control problems. 

The paper is organized into several sections. In Section 2 we set the problem by describing the two component domain $\Omega$ and the functional framework. Then, we state our main theorem describing the homogenization result. In Section 3 we show existence, uniqueness and a uniform a priori estimate of the solution to the fine-scale problem. In Section 4 we briefly recall the definition and the properties of the unfolding operators in domains with inclusions by adopting the same notations as in \cite{donngu}. Finally, in Section 5 we detail the proof of the main result according to the different values of the parameter $\gamma$.

\section{Setting of the problem and main result}\label{secsetting}

\subsection{The two-component domain} 
Let $\o$ be a bounded open set in $\R^N$, $N\geq 2$, with a Lipschitz-continuous boundary $\partial \o$. We denote by $Y\doteq \prod_{i=1}^N [0,l_i[$ the reference cell, with $l_i>0, i=1, ..., N$. We assume that $Y\doteq Y_1\cup \overline{Y_2}$, where $Y_1$ and $Y_2$ are two disjoint connected open sets such that $Y_2\neq \emptyset$, $\overline{Y_2}\subset Y$ and the common boundary $\G\doteq\partial Y_2$ is Lipschitz-continuous.\\
For any $k \in \Z^N$, let $k_l\doteq(k_1 l_1, ..., k_N l_N)$ and
$$Y^k\doteq k_l+Y, \qquad Y^k_i\doteq k_l +Y_i, \qquad i=1,2.$$
\noindent
Given a positive parameter $\e$ taking values in a sequence converging to zero, we set:\vskip1mm
$$J_\e\doteq \{k\in\Z^N |\ \e \overline{Y^k_2}\subset \o\}$$
 and
$$\o_2^\e\doteq \bigcup_{k\in J_\e} \e Y^k_2, \qquad \o_1^\e\doteq \o\setminus \overline{\o_2^\e}, \qquad \Ge\doteq\partial \o_2^\e.$$
Therefore, one has $\o=\o_1^\e \cup \o_2^\e \cup \Ge$. Moreover, for the sake of simplicity, we suppose that the second component doesn't meet the boundary of the domain.%   , see Figure \ref{figD}.\\
%\vfill
%\begin{figure}[ht]
%\begin{center}
%\includegraphics[width=0.7\linewidth]{2comp}
%\end{center}
%\vskip -10 true mm
%\caption{The two-component domain $\o$ and the reference cell $Y$}
%\label{figD}
%\end{figure}

Throughout the paper 
\begin{itemize}
\item $\displaystyle \mathcal{M}_E(v)$ is the average on $E$ of a function $v \in L^1(E)$, where  $E$ is an open subset of $ \R^{N-1}$ or  $ \R^N$;
\item $\displaystyle\chi\strut_{\omega}$ is the characteristic function of any open set $\omega$ of $\R^N$;
\item $\sim$ is the zero extension to the whole of $\o$ of functions defined on $\o^\e_1$ or $\o^\e_2$;
\item $\theta_i\doteq\frac{|Y_i|}{|Y|}$, $i=1,2$;
\item $C$ denotes different positive constants independent of $\e$;
\item $\mathcal{M}(\alpha, \beta, Y)$ is the set of the $Y$-periodic matrix fields $A=(a_{i,j})_{1 \leq i,j \leq N} \in (L^{\infty}(Y))^{N \times N}$ such that 
$(A(x)\lambda,\lambda)\geq \alpha |\lambda|^2$ and $|A(x)\lambda|\leq \beta|\lambda|, \ \forall \lambda \in \R^N$ and a.e. in $Y$,
with $\alpha, \beta \in \R$, $0<\alpha<\beta$;
\item $\{ e_1, ..., e_N \}$ is the canonical basis of $\R^N$;
\item $[v]$ denotes the jump through $\gae$ of a function $v=(v_1,v_2)$ where $v_i$ is defined in $\o_i^\e$, $i=1,2$.
\end{itemize} 
\subsection{The problem}
In the domain previously described, we consider the following problem:
\begin{equation}\label{prob}
\left\{\begin{array}{lllll}
\displaystyle
-\operatorname{div}(A^\e\n \uoe)=f & \hbox{ in } \o_1^\e, \\[2mm]
\displaystyle
-\operatorname{div}(A^\e\n \ute)=f & \hbox{ in } \o_2^\e, \\[2mm]
\displaystyle
(A^\e\n \uoe)\cdot\nu_1^\e=-(A^\e\n \ute)\cdot\nu_2^\e & \hbox{ on } \Ge,\\[2mm]
\displaystyle
[\ue]\geq 0, \quad (A^\e\n \uoe)\cdot\nu_1^\e+\e^\g h^\e[\ue] \geq 0 & \hbox{ on } \Ge,\\[2mm]
\displaystyle
[\ue]((A^\e\n \uoe)\cdot\nu_1^\e+\e^\g h^\e[\ue])=0 & \hbox{ on } \Ge,\\[2mm]
\displaystyle
\uoe =0 & \hbox{ on } \partial \Omega,
\end{array}\right.
\end{equation}
where $\g\in \R$ and $\nu_i^\e$ is the unit external normal vector to $\o_i^\e$, $i=1,2$.

Our aim is to study the asymptotic behavior, as $\e$ goes to zero, of problem \eqref{prob} under the following assumptions on the data:
\vskip 1mm
(\textbf{A$1$}) \quad For any $\e>0$, the coeffficients matrix $A^\e(x)\doteq A\left(\frac{x}{\e}\right)$ a.e. in $\o$, where $A$ belongs to $\mathcal{M}(\alpha, \beta, Y)$;
\vskip 1mm
(\textbf{A$2$})\quad $\gamma\leq 1$;
\vskip 1mm
(\textbf{A$3$})\quad $f\in L^2(\Omega)$;
\vskip 1mm
(\textbf{A$4$})\quad $h^\e$ is defined by $h^\e(x)\doteq h(\frac{x}{\e})$ a.e. on $\gae$, where $h$ is a $Y$-periodic function in $L^{\infty}(\G)$ such that
$$\hbox{there exists } h_0 \in \R :  0 < h_0 < h(y) \hbox{ a.e.  on }\G.$$

Let us remark that we assume hypothesis (\textbf{A$2$}) since, in \cite{hum}, H.C. Hummel showed that one cannot expect uniformly bounded solutions when $\g>1$.

\subsection{The functional framework and the weak formulation}
Let $V^\e\doteq \{v \in H^1(\o_1^\e) \ |\ v=0 \hbox{ on } \partial \o\}$ endowed with the norm
$$\|v\|_{V^\e}\doteq \|\n v\|_{L^2(\o_1^\e)}.$$
It is known (see for instance \cite[Lemma $1$]{cp}) that a Poincar\'e inequality in $V^{\e}$ holds with a constant $C_P$ independent of $\e$, that is
\begin{equation}\label{Poi}
\|v\|_{L^2(\o^\e_1)}\leq C_P\|\n v\|_{L^2(\o^\e_1)} \quad \forall v \in V^\e.
\end{equation}
Consequently, the norm in $V^\e$ is equivalent to the one in $H^1(\o_1^\e)$ via a constant independent of $\e$.\\
Let $H^\e_\g$ be the space defined by
$$H^\e_\g\doteq \{v=(v_1,v_2)|\ v_1 \in V^\e,\ v_2 \in H^1(\o_2^\e)\},$$
which, after the identification $\n v:=  \widetilde{\n v_1} + \widetilde{\n v_2}$, is a Hilbert space
equipped with the norm
$$\|v\|^2_{H^\e_\g}\doteq\|\n v_1\|^2_{L^2(\o_1^\e)}+ \|\n v_2\|^2_{L^2(\o_2^\e)} + \e^\g\|[v]\|^2_{L^2(\Ge)}.$$
\begin{proposition}[\cite{DoJo}\cite{DoLNTa}\cite{Mo}]\label{eqnorm}
Let $\gamma\leq 1$. There exist two positive constants $C_1$ and $C_2$, independent of $\e$, such that 
$$\forall v \in H^\e_\g, \qquad C_1 \|v\|^2_{V^\e \times H^1(\o_2^\e)}\leq \|v\|^2_{H^\e_\g}\leq C_2 (1+\e^{\g-1})\|v\|^2_{V^\e \times H^1(\o_2^\e)}.$$
Therefore, if $v^\e=(v^\e_1,v^\e_2)$ is a bounded sequence in $H^\e_\g$, then
\begin{equation*}
\begin{array}{lll}
\displaystyle
\|v^\e_1\|_{H^1(\o^\e_1)}\leq C,\\[3mm]
\displaystyle
\|v^\e_2\|_{H^1(\o^\e_2)}\leq C,\\[3mm]
\displaystyle
\|[v]\|_{L^2(\G^\e)}\leq C \e^{-\frac{\gamma}{2}},
\end{array}
\end{equation*}
with $C$ positive constant independent of $\e$.
\end{proposition}
Let us define the following closed and convex subset of $H^\e_\g $:
\beq\label{keg}
\ds
K^\e_\g\doteq \{v \in H^\e_\g \ |\ [v]\geq 0 \hbox{ on } \gae\}.
\eeq
Hence, the weak formulation of problem \eqref{prob} is given by the variational inequality
\beq\label{pe}
\left\{\ba{llll}
\ds \text{Find } \ue \in K^\e_\g \text{ such that }\\[2mm]
\ds \int_{\o_1^\e} A^\e\n \uoe\n (v_1-\uoe)dx+\int_{\o_2^\e} A^\e\n \ute\n (v_2-\ute)dx +\e^\g\int_{\gae}h^\e[\ue]([v]-[\ue])d\sigma\\[5mm]
\ds \geq \int_{\o_1^\e} f(v_1-\uoe)dx+\int_{\o_2^\e} f(v_2-\ute)dx,\quad \forall v \in K^\e_\g.
\ea\right.
\eeq
Indeed, by multiplying the first equation in problem \eqref{prob} by $\ue$ and integrating by parts, one obtains
\beq
\ds \int_{\o_1^\e} A^\e\n \uoe\n\uoe\, dx+\int_{\o_2^\e} A^\e\n \ute\n\ute\, dx -\int_{\gae}(A^\e\n \ue)\cdot\nu_1^\e[\ue]d\sigma=\int_{\o_1^\e} f\uoe\, dx+\int_{\o_2^\e} f\ute\, dx.
\eeq
Thus, adding and subtracting $\e^\g\displaystyle\int_{\gae}h^\e[\ue]^2\,d\sigma$, we get
\beq
\ba{ll}
\ds \int_{\o_1^\e} A^\e\n \uoe\n\uoe\, dx +\int_{\o_2^\e} A^\e\n \ute\n\ute\, dx -\int_{\gae}((A^\e\n \ue)\cdot\nu_1^\e+\e^\g h^\e[\ue])[\ue]d\sigma+\e^\g\int_{\gae}h^\e[\ue]^2\,d\sigma\\[5mm]
\ds =\int_{\o_1^\e} f\uoe\, dx+\int_{\o_2^\e} f\ute\, dx,
\ea
\eeq
which leads to 
\beq\label{equa}
\ds \int_{\o_1^\e} A^\e\n \uoe\n\uoe\, dx +\int_{\o_2^\e} A^\e\n \ute\n\ute\, dx +\e^\g\int_{\gae}h^\e[\ue]^2\,d\sigma=\int_{\o_1^\e} f\uoe\, dx+\int_{\o_2^\e} f\ute\, dx,
\eeq
in view of the fifth condition in \eqref{prob}.\\
On the other hand, by multiplying the first equation in problem \eqref{prob} by a function $v\in K^\e_\g$, integrating by parts and taking into account the fourth condition in \eqref{prob}, one has
\beq\label{diseq}
\ba{lll}
\ds \int_{\o_1^\e} A^\e\n \uoe\n v_1\, dx + \int_{\o_2^\e} A^\e\n \ute\n v_2\, dx +\e^\g\int_{\gae}h^\e[\ue][v]\,d\sigma\\[5mm]
\ds =\int_\o fv\, dx+\int_{\gae}((A^\e\n \ue)\cdot\nu_1^\e+\e^\g h^\e[\ue])[v]\,d\sigma\geq \int_{\o_1^\e} f v_1\, dx+\int_{\o_2^\e} f v_2\, dx.
\ea
\eeq
Finally, the variational formulation \eqref{pe} is obtained by subtracting \eqref{equa} from \eqref{diseq}.

In Section \ref{secexuniq} we will prove that, for any fixed $\e$, problem \eqref{pe} admits a unique solution. 
\subsection{The main result}
The following theorem, which is the main result of this paper, shows that the asymptotic behaviour of the weak solution of problem \eqref{prob} varies according to the different values of the parameter $\gamma$.
\begin{theorem}\label{maintheo}
Under assumptions $(\textbf{A}1)-(\textbf{A}4)$, let $\ue=(\uoe,\ute)$ be the weak solution of problem \eqref{prob}.
\begin{itemize}
\item[$\blacktriangleright$] {\bf Case $\g<1$}\\
There exists $u_1\in H^1_0(\o)$ such that
\beq
\ds\widetilde{\ue_i}\ru \theta_i u_1 \quad \hbox{ weakly in } L^2(\o),\ i=1,2,
\eeq
where $u_1$ is the unique solution of problem
\beq\label{probhom}
\left\{\ba{llll}
\ds
-\operatorname{div}(A^0_\g(\n u_1))=f & \hbox{ in } \o, \\[2mm]
\ds
u_1 =0 & \hbox{ on } \partial \o.
\ea\right.
\eeq
\begin{itemize}
\item[$\bullet$] If $\g<-1$, $$A^0_\g(\n u_1)=A^0_\g\n u_1$$
where the constant matrix $A^0_\g$ is given by
\beq\label{A0g}
A^0_\g=A^1_\g+A^2_\g
\eeq
being
\beq
A^l_\g\doteq \left\{\theta_l\mathcal{M}_{Y_l}\left(a_{ij}-\sum_{k=1}^N a_{ik}\frac{\partial \overline{\chi}_j}{\partial y_k}\right)\right\}_{N\times N},\quad l=1,2
\eeq
with $\overline{\chi}=\left(\overline{\chi}_1,\ldots,\overline{\chi}_N\right)$, and $\overline{\chi}_j$ is the solution of the cell problem
\beq\label{cell2c}
\left\{\ba{llll}
\ds
-\operatorname{div}(A \n (\overline{\chi}_j-y_j))=0 \quad \hbox{ in } Y, \\[2mm]
\ds\overline{\chi}_j \quad Y-periodic, \quad \mathcal{M}_\G(\overline{\chi}_j)=0.
\ea\right.
\eeq
Moreover it holds
\beq
A^{\varepsilon}\widetilde{\nabla u^{\varepsilon}_l}\ru A^l_\g\nabla u_1\quad\hbox{ weakly in } (L^2(\o))^N,\quad l=1,2.
\eeq
\item[$\bullet$] If $\g=-1$,
\beq\label{ahom}
A^0_{-1}(\z)\doteq\sum_{i=1}^2\frac{1}{|Y|}\int_{Y_i}A(y)\left(\z+\n_y \widehat{\chi}_i(y,\z)\right)dy,\quad \z\in \R^N,
\eeq
with $({\widehat{\chi}}_1(\cdot,\z),{\widehat{\chi}}_2(\cdot,\z))$ solution of the following cell problem:
\beq\label{celltircase}
\left\{\ba{llll}
\ds \text{Find } ({\widehat{\chi}}_1 (\cdot,\z),{\widehat{\chi}}_2 (\cdot,\z)) \in \overline{W} \text{ such that }\\[2mm]
\ds \frac{1}{|Y|}\int_{Y_1} A(y)(\z+\n_y {\widehat{\chi}}_1 (y,\z))(\n_y z_1-\n_y {\widehat{\chi}}_1 (y,\z))dy\\[5mm]
\ds+\frac{1}{|Y|}\int_{Y_2} A(y)(\z+\n_y {\widehat{\chi}}_2 (y,\z))(\n_y z_2-\n_y {\widehat{\chi}}_2 (y,\z))dy\\[5mm]
\ds +\frac{1}{|Y|}\int_{\G}h(y)({\widehat{\chi}}_1 (y,\z)-{\widehat{\chi}}_2 (y,\z))((z_1-z_2)-({\widehat{\chi}}_1 (y,\z)-{\widehat{\chi}}_2 (y,\z)))\,dxd\sigma_y \geq 0,\\[5mm]
\ds \forall (z_1,z_2)\in H^1_{per}(Y_1)\times H^1(Y_2):\ z_1 \geq z_2 \hbox{ on }\G,
\ea\right.
\eeq
where 
\beq \label{Wbar}
\ds \overline{W}\doteq\{(z_1,z_2)\in H^1_{per}(Y_1)\times H^1(Y_2)|\mathcal{M}_\G(z)=0 \hbox{ and } z_1 \geq z_2 \hbox{ on }\G\}.
\eeq
Moreover it holds
\beq
A^{\varepsilon}\widetilde{\nabla u^{\varepsilon}_l}\ru \theta_l\mathcal{M}_{Y_l}\left(A(y)\left(\nabla u_1+\n_y \widehat{\chi}_l(y,\nabla u_1)\right)\right)\quad\hbox{ weakly in } (L^2(\o))^N,\quad l=1,2.
\eeq
\item[$\bullet$] If $\g\in ]-1,1[$, $$A^0_\g(\n u_1)=A^0_\g\n u_1$$
where the constant matrix $A^0_\g$ is given by
\beq\label{a0g}
A^0_\g\doteq \left\{\theta_1\mathcal{M}_{Y_1} \left(a_{i,j}-\sum_{k=1}^N a_{i,k}\frac{\partial \chi_j}{\partial y_k}\right)\right\}_{N\times N},
\eeq
with $\chi=(\chi_1, \ldots, \chi_N)$, being $\chi_j$ the solution of the cell problem
\beq\label{cell1c}
\left\{\ba{llll}
\ds
-\operatorname{div}(A(y) \n (\chi_j-y_j))=0 & \hbox{ in } Y_1, \\[2mm]
\ds (A(y) \n (\chi_j-y_j))\cdot \nu_1 =0 & \hbox{ on } \G,\\[2mm]
\ds\chi_j \quad Y-periodic, \qquad \mathcal{M}_{Y_1}(\chi_j)=0.
\ea\right.
\eeq
\end{itemize}
Moreover it holds
\beq\label{flusso}
\left\{
\begin{array}{ll}
A^{\varepsilon}\widetilde{\nabla u^{\varepsilon}_1}\ru A^0_\g\nabla u_1&\hbox{ weakly in } (L^2(\o))^N,\\
A^{\varepsilon}\widetilde{\nabla u^{\varepsilon}_2}\ru 0&\hbox{ weakly in } (L^2(\o))^N.
\end{array}
\right.
\eeq
\item[$\blacktriangleright$] {\bf Case $\g=1$}\\
There exist $(u_1,u_2)\in H^1_0(\o)\times L^2(\o)$ such that
\beq
\ds\widetilde{\ue_i}\ru \theta_i u_i \quad \hbox{ weakly in } L^2(\o),\ i=1,2,
\eeq
and $(u_1,u_2)$ is the unique solution of problem
\beq\label{obstacle}
\left\{\ba{lll}
\ds \hbox{Find }(u_1,u_2)\in K \hbox{ s.t. }\\[2mm]
\ds \int_{\o} A^0_1 \n u_1(\n \f_1-\n u_1)\, dx +\frac{\mathcal{M}_\G(h)}{|Y|}\int_{\o} (u_1-u_2)[(\f_1-\f_2)-(u_1-u_2)]\, dx \\[5mm]
\ds \geq \theta_1\int_\o f(\f_1-u_1)\, dx+\theta_2\int_\o f(\f_2-u_2)\, dx,\ \forall (\f_1,\f_2)\in K,
\ea\right.
\eeq
where $K\doteq\{(\f_1,\f_2)\in H^1_0(\o)\times L^2(\o) \,|\, \f_1\geq \f_2 \hbox{ in }\o\}$ and the constant matrix $A^0_1$ is as in \eqref{a0g}.\\
Moreover, it holds \eqref{flusso} with $\gamma=1$.
\end{itemize} 
\end{theorem}
\begin{remark}\label{commento}
We explicitly observe that the most interesting cases are $\gamma=-1$ and $\gamma=1$. Actually, for $\gamma=-1$, the homogenized problem \eqref{probhom} is nonlinear. In particular, the associated operator is defined through the solution of a cell problem  consisting in a variational inequality presenting a jump on the interface $\G$ (cf. \eqref{ahom}-\eqref{celltircase}). Here, the positivity constraint of the jump on $\Ge$ gives rise to a positive jump on $\G$ in the cell problem.\\ 
When $\gamma=1$,  the Signorini-type interface condition at $\e$ level leads to an obstacle-type homogenized problem.\\
In the remaining cases the homogenized problems are linear and standard (see \cite{DoLNTa}, \cite{DoMo}, \cite{Mo}). Hence, the Signorini-type interface condition in the $\e$-problem does not affect the limit behaviour.  
\end{remark}
\vskip 1cm
\section{Existence and uniqueness result}\label{secexuniq}
We prove here the existence and the uniqueness of the solution of problem \eqref{pe},  for any fixed $\ep$, together with some uniform a priori estimates. 
\begin{theorem}\label{exun}
Under assumptions $(\textbf{A}1)-(\textbf{A}4)$,  for any fixed $\ep$, there exists a unique solution $u^\ep\in K^\e_\g$ of  problem \eqref{pe} satisfying 
\begin{equation}\label{estsol}
\|\ue\|_{H^\e_\g}\leq C,
\eeq
with $C$ positive constant independent of $\ep$.
\end{theorem}
\begin{proof}
For any fixed $\ep$, we want to apply \cite[Proposition $2.5$ (pag. 179), Theorems $8.2-8.3$ (pag. $248$)]{lions}. To this aim, we consider the operator
$$ \mathcal{A}_\e: H^\e_\g \rw (H^\e_\g)'$$
such that
\beq\label{Aep}
\ds \langle\mathcal{A}_\e(u),v\rangle_{(H^\e_\g)',H^\e_\g}\doteq \int_{\o_1^\e} A^\e\n u_1\n v_1 dx+\int_{\o_2^\e} A^\e\n u_2\n v_2 dx +\e^\g\int_{\gae}h^\e[u][v]d\sigma,
\eeq
where $(H^\e_\g)'$ denotes the dual space of $H^\e_\g$.\\
Hence, problem \eqref{pe} can be rewritten as 
\beq\label{pe*}
\left\{\ba{lll}
\ds \text{Find } \ue \in K^\e_\g \text{ such that }\\[2mm]
\ds \langle\mathcal{A}_\e(\ue),v-\ue\rangle_{(H^\e_\g)',H^\e_\g}\geq \int_{\o_1^\e} f(v_1-\uoe)dx+\int_{\o_2^\e} f(v_2-\ute)dx,\quad \forall v \in K^\e_\g.
\ea\right.
\eeq

The hemicontinuity, the equi-boundedness  and the strict monotonicity of $\mathcal{A}_\e$ are immediate consequences of $(\textbf{A}1)$ and $(\textbf{A}4)$. Therefore, by \cite[Proposition $2.5$ (pag. 179)]{lions} one has that $\mathcal{A}_\e$ is pseudo-monotone.\\
Moreover, it is easy to check that $\mathcal{A}_\e$ is also equicoercive. Hence, the existence and the uniqueness of the solution follow from \cite[Theorems $8.2-8.3$]{lions}.

Finally, we prove the a priori estimates by choosing $v=0$ as test function in \eqref{pe*}. One gets
$$\langle\mathcal{A}_\e(\ue),-\ue\rangle_{(H^\e_\g)',H^\e_\g}\geq \int_{\o_1^\e} f(-\uoe)dx+\int_{\o_2^\e} f(-\ute)dx.$$
Using again $(\textbf{A}1)$ and $(\textbf{A}4)$ and applying the Holder inequality, we have
\beq
\ds  \alpha \|(\n\uoe,\n\ute)\|^2_{L^2(\o_1^\e)\times L^2(\o_2^\e)} + \e^\g h_0\|[\ue]\|^2_{L^2(\gae)}\leq \langle\mathcal{A}_\e(\ue),\ue\rangle_{(H^\e_\g)',H^\e_\g} \leq (\|\uoe\|_{L^2(\o_1^\e)} +\|\ute\|_{L^2(\o_2^\e)})\|f\|_{\lto}.
\eeq
In view of \eqref{Poi} and Proposition \ref{eqnorm}, the above inequality leads to
$$
\displaystyle
\alpha \|(\n\uoe,\n\ute)\|^2_{L^2(\o_1^\e)\times L^2(\o_2^\e)} + \e^\g h_0\|[\ue]\|^2_{L^2(\gae)} \leq C(C_p, C_1) \|\ue\|_{H^\e_\g},
$$
with $C$ positive constant independent of $\e$. This gives the result.
\end{proof}
\section{The periodic unfolding method}\label{method}
In this section we recall the definitions and the properties of the periodic unfolding operators for two-component domains introduced by P. Donato, K.H. Le Nguyen and R. Tardieu in \cite{DoLNTa}. Throughout the paper, we use the same notation of \cite{donngu}.

Let $z \in \R^N$, we denote by $[z]_Y$ its integer part such that $z-[z]_Y$ belongs to $Y$ and set $\{z\}_Y\doteq z-[z]_Y$. Then, for every positive $\e$,
$$x=\e\left(\left[\frac{x}{\e}\right]_Y + \left\{\frac{x}{\e}\right\}_Y\right) \ \ \forall x \in \R^N.$$
We introduce the following sets:
\begin{equation*}
\widehat{J}_\e\doteq\{k\in\Z^N |\ \e Y^k\subset \o\},\ \ \widehat{\o}_\e\doteq\hbox{interior}\left\{\bigcup_{k\in \widehat{J}_\e} \e (k_l +\overline{Y})\right\},\ \ \Lambda_\e\doteq\o\setminus \widehat{\o}_\e,
\end{equation*}
and
\begin{equation*}
\widehat{\o}_i^\e\doteq\bigcup_{k\in \widehat{J}_\e} \e Y^k_i, \qquad \Lambda_i^\e\doteq\o_i^\e\setminus \widehat{\o}_i^\e, \quad i=1,2, \qquad \widehat{\G}^\e\doteq\partial \widehat{\o}_2^\e.
\end{equation*}
\begin{definition}
For any Lebesgue-measurable function $\phi$ on $\o^\e_i$, $i=1,2$, we set
\begin{equation}\label{uo}
\te_i(\phi)(x,y)\doteq\begin{cases}
\phi\left(\e\left[\frac{x}{\e}\right]_Y + \e y\right) & \hbox{ a.e. for } (x,y) \in \widehat{\o}_\e\times Y_i,\\
0 & \hbox{ a.e. for } (x,y) \in \Lambda_\e \times Y_i.
\end{cases}
\end{equation}
\end{definition}
We observe that $\te_1$ is the unfolding operator defined in \cite{d} for perforated domains, while $\te_2$ is the one introduced in \cite{DoLNTa} for the inclusions. If $\phi$ is defined in $\o$, we simply write $\te_i(\phi)$ instead of $\te_i(\phi_{|\o_i^\e})$, $i=1,2$. 

%\rd LA USIAMO? ALTRIMENTI ELIMINA E METTI 'DEFINITIONS OF TWO OPERATORS' ALL'INIZIO DEL PARAGRAFO\bk
%\begin{definition} For any Lebesgue-measurable function $\phi$ on $\o$, the unfolding operator for two-component domain $\te(\phi)$ is defined as follows:
%\begin{equation*}
%\te(\phi)(x,y)\doteq\begin{cases}
%\te_1(\phi) & \hbox{ in } \o\times Y_1,\\
%\te_2(\phi) & \hbox{ in } \o \times Y_2.
%\end{cases}
%\end{equation*}
%\end{definition}

We now recall the main properties of the unfolding operators.
\begin{proposition}[\cite{d}\cite{DoLNTa}]\label{pu}
Let $p \in [1,+\infty[$ and $i=1,2$.
\begin{enumerate}
\item $\te_i$ is a linear and continuous operator from $L^p(\o^\e_i)$ to $L^p(\o\times Y_i)$.
\item $\te_i(\phi\psi)=\te_i(\phi)\te_i(\psi)$ for every $\phi, \psi \in L^p(\o^\e_i)$.
\item Let $\phi \in L^p(Y_i)$ be a $Y$-periodic function and set $\phi_\e(x)=\phi(\frac{x}{\e})$.\\
Then, $$\te_i(\phi_\e)(x,y)=\begin{cases}
\phi(y) & \hbox{ a.e. for } (x,y) \in \widehat{\o}_\e\times Y_i,\\
0 & \hbox{ a.e. for } (x,y) \in \Lambda_\e \times Y_i,
\end{cases}$$ and
$$\te_i(\phi_\e)\rw\phi \quad \hbox{ strongly in }L^p(\o \times Y_i).$$
\item For all $\phi \in L^1(\o^\e_i)$, one has
$$\frac{1}{|Y|}\int_{\o\times Y_i} \te_i(\phi)(x,y) dxdy=\int_{\widehat{\o}^\e_i} \phi(x) dx=\int_{\o^\e_i} \phi(x) dx-\int_{\Lambda^\e_i} \phi(x) dx.$$
\item $\|\te_i(\phi)\|_{L^p(\o \times Y_i)}\leq |Y|^{\frac{1}{p}} \|\phi\|_{L^p(\o^\e_i)}$ for every $\phi \in L^p(\o^\e_i)$.
\item For $\phi \in L^p(\o)$, \ $\te_i(\phi) \rw \phi \quad \hbox{strongly in }L^p(\o \times Y_i)$.
\item Let $\{\phi_\e\}$ be a sequence in $L^p(\o)$ such that $\phi_\e \rw \phi \ \hbox{strongly in }L^p(\o)$. Then
$$\te_i(\phi_\e) \rw \phi \quad \hbox{strongly in }L^p(\o\times Y_i).$$
\item Let $\{\phi_\e\}$ be a sequence in $L^p(\o_i^\e)$ such that $\|\phi_\e\|_{L^p(\o_i^\e)}\leq C$.\\ If $\te_i(\phi_\e)\ru \widehat{\phi} \hbox{ weakly in }L^p(\o\times Y_i)$, then
$$\widetilde{\phi_\e} \ru \theta_i\mathcal{M}_{Y_i}(\widehat{\phi}) \quad \hbox{weakly in }L^p(\o).$$
\item If $\phi \in W^{1,p}(\o_i^\e)$, then $\n_y[\te_i(\phi)]=\e\te_i(\n \phi)$ and $\te_i(\phi) \in L^p(\o, W^{1,p}(Y_i))$.
\end{enumerate}
\end{proposition}
We give below the main propositions concerning the jump on the interface and some convergence results.\\
Arguing as in the proof of Lemma 2.14 of \cite{DoLNTa}, we can easily obtain the following result:
\begin{lemma}
Let $u^\e \in H^\e_\g$ and $h$ satisfy (\textbf{A$4$}). It holds
\beq\label{brd*}
\e \int_{\Ge}h^\e(\ue_1-\ue_2)^2d\sigma_x\geq \frac{1}{|Y|}\int_{\o \times \G} h(y)(\te_1(\ue_1)-\te_2(\ue_2))^2dxd\sigma_y.
\eeq
\end{lemma}
Moreover, one has:
\begin{lemma}\cite{DoLNTa}\label{phibordo}
Let $u^\e \in H^\e_\g$ and $h$ satisfy (\textbf{A$4$}). If $\phi\in \d(\o)$, then for $\e$ small enough one has
\beq\label{brd}
\e \int_{\Ge}h^\e(\ue_1-\ue_2)\phi d\sigma_x=\frac{1}{|Y|}\int_{\o \times \G} h(y)(\te_1(\ue_1)-\te_2(\ue_2))\te_1(\phi)dxd\sigma_y.
\eeq
\end{lemma}
We note that in the right hand sides of \eqref{brd*} and \eqref{brd}, the traces are well defined in view of Proposition \ref{pu}$_9$.
\begin{theorem}\cite{up5}\cite{d}\cite{donngu}\cite{DoLNTa}\label{teconv}
Let $\g\in \R$ and $\ue$ be a bounded sequence in $H^\e_\g$. Then,
\begin{equation}\label{tediff}
\begin{array}{lll}
\displaystyle
\|\te_1(\n \ue_1)\|_{L^2(\o \times Y_1)}\leq C,\\[1mm]
\displaystyle
\|\te_2(\n \ue_2)\|_{L^2(\o \times Y_2)}\leq C,\\[1mm]
\displaystyle
\|\te_1(\ue_1)-\te_2(\ue_2)\|_{L^2(\o \times \G)}\leq C\e^{\frac{1-\g}{2}},
\end{array}
\end{equation}
and there exist a subsequence (still denoted by $\e$), $u_1\in H^1_0(\o)$ and $\widehat{u}_1\in L^2(\o, H^1_{per}(Y_1))$ with $\mathcal{M}_{\G}(\widehat{u}_1)=0$ a.e. in $\o$ such that
\begin{equation}\label{te1}
\left\{\begin{array}{ll}
\displaystyle
\te_1(\ue_1)\rw u_1 & \hbox{ strongly in }L^2(\o, H^1(Y_1)),\\[1mm]
\displaystyle
\te_1(\n\ue_1)\ru \n u_1 +\n_y \widehat{u}_1 & \hbox{ weakly in }L^2(\o\times Y_1).
\end{array}\right.
\end{equation}
Moreover, if $\g\leq 1$, there exist a subsequence (still denoted by $\e$), $u_2\in L^2(\o)$ and $\overline{u}_2\in L^2(\o, H^1(Y_2))$ with $\mathcal{M}_{\G}(\overline{u}_2)=0$ a.e. in $\o$ such that
\begin{equation}\label{te2}
\left\{\begin{array}{ll}
\displaystyle
\te_2(\ue_2)\ru u_2 & \hbox{ weakly in }L^2(\o, H^1(Y_2)),\\[1mm]
\displaystyle
\te_2(\n\ue_2)\ru \n_y \overline{u}_2 & \hbox{ weakly in }L^2(\o\times Y_2).
\end{array}\right.
\end{equation}
Furthermore, if $\g <1$, then $$u_1=u_2.$$
In addition, for $\g\leq-1$, there exists $\xi_\G\in L^2(\Omega)$ such that, set 
\beq\label{u2hat}
\widehat{u}_2\doteq\overline{u}_2-y_\G\n u_1-\xi_\G\in L^2(\o,H^1(Y_2))
\eeq
with $y_\G\doteq y-\mathcal{M}_\G(y)$, one has:
\begin{itemize}
\item if $\g<-1$, 
\beq\label{u1hat}
\widehat{u}_1=\widehat{u}_2+\xi_\G \quad \hbox{ on }\Omega \times \G;
\eeq
\item if $\g=-1$, $$\ds\frac{\toe(\uoe)-\tte(\ute)}{\e}\ru \widehat{u}_1-\widehat{u}_2 \quad \hbox{ weakly in }L^2(\Omega \times \G).$$
\end{itemize} 
\end{theorem}
\section{Proof of the homogenization result}
In this section we prove the convergence results stated in Theorem \ref{maintheo}. The proof differs according to the values of the parameter $\g$.

For $\g<-1$, by choosing suitable test functions, it is possible to adapt the arguments of \cite{DoLNTa} leading to the classical homogenized problem of Bensoussan-Lions-Papanicolau in a fixed domain, see \cite{ben}.

For $\g=-1$, the above mentioned arguments do not apply. Indeed, in this case, the choice of the test functions leads to handle products of two weakly convergent sequences. We can overcome this difficulty by using lower semi-continuity arguments, due to the fact that we are dealing with variational inequalities. In addition, the constraint imposed on the chosen test functions gives rise to a nonlinear homogenized system where the cell problem consists in a variational inequality. 

For the cases $\g\in]-1,1[$ and $\g=1$, in order to analyse the asymptotic behaviour of the problem at microscopic level, we need to consider more regular test functions.\\
As far as it concerns the homogenized problem, in the case $\g\in]-1,1[$, it turns out to be the classical limit problem of Cioranescu-S.J.Paulin in perforated domains, see \cite{cp}.\\
While, when $\g=1$, again taking advantage of lower semi-continuity arguments, we obtain an obstacle-type homogenized problem.

\subsection{The case $\g< -1$}
Let $u^\e$ be the weak solution of problem \eqref{prob}. Theorem \ref{teconv}, for the case $\g<-1$, ensures the existence of $u_1\in H^1_0(\o)$, $\widehat{u}_1\in L^2(\o, H^1_{per}(Y_1))$ with $\mathcal{M}_{\G}(\widehat{u}_1)=0$ a.e. in $\o$ and $\widehat{u}_2\in L^2(\o, H^1(Y_2))$ such that, up to a subsequence,
\beq\label{convseccase}
\left\{\ba{lllll}
\ds \widetilde{u^\e_i} \ru \theta_i u_1 & \hbox{ weakly in } L^2(\o),\ i=1,2, \\[3mm]
\ds\te_1(\ue_1)\rw u_1 & \hbox{ strongly in } L^2(\o, H^1(Y_1)),\\[3mm]
\ds \te_2(\ue_2)\ru u_1 & \hbox{ weakly in }L^2(\o, H^1(Y_2)),\\[3mm]
\ds\te_1(\n\ue_1)\ru \n u_1 +\n_y \widehat{u}_1 & \hbox{ weakly in }L^2(\o\times Y_1),\\[3mm]
\ds \te_2(\n\ue_2)\ru \n u_1+\n_y \widehat{u}_2 & \hbox{ weakly in }L^2(\o\times Y_2).
\ea\right.
\eeq
We first take $\f\in\d(\o)$ and use $(\uoe +\f, \ute+\f)\in K^\e_\g$ as test function in \eqref{pe}, obtaining
\beq\label{star}
\ds\int_{\o_1^\e} A^\e\n \uoe\n \f dx+\int_{\o_2^\e} A^\e\n \ute\n \f dx \geq \int_{\o_1^\e} f\f dx+\int_{\o_2^\e} f\f dx, \quad \forall \f \in \d(\o).
\eeq
By (\textbf{A$1$}) and \eqref{estsol} one has, for $i=1,2$,
\beq
\int_{\Lambda_i^\e} A^\e\n \ue_i \n \f dx\leq \beta \|\n\ue_i\|_{L^2(\o_i^\e)}\|\n \f\|_{L^2(\Lambda_i^\e)}\leq\beta C\|\n \f\|_{L^2(\Lambda_i^\e)}\rw 0, \quad \e \rw 0,
\eeq
in view of the equicontinuity of the integral. Consequently, by Proposition \ref{pu}$_4$, inequality \eqref{star} can be rewritten as
\beq
\ba{cc}
\ds\frac{1}{|Y|}\int_{\o\times Y_1}\toe(A^\e)\toe(\n \uoe)\toe(\n \f) dxdy + \frac{1}{|Y|}\int_{\o\times Y_2}\tte(A^\e)\tte(\n \ute)\tte(\n \f) dxdy\\[5mm] \ds \geq \int_{\o_1^\e} f\f dx+\int_{\o_2^\e} f\f dx, \quad \forall \f \in \d(\o).
\ea
\eeq
We pass to the limit, as $\e$ tends to zero, in the previous inequality by using Proposition \ref{pu}$_3$ and \eqref{convseccase},  we get
\beq
\ds\frac{1}{|Y|}\int_{\o\times Y_1} A(y) (\n u_1+\n_y \widehat{u}_1)\n \f\, dx dy + \frac{1}{|Y|}\int_{\o\times Y_2} A(y)(\n u_1+\n_y \widehat{u}_2)  \n \f\, dxdy \geq  \int_\o f\f\, dx, \quad \forall \f \in \d(\o).
\eeq
The validity of the previous inequality for every function in $\d(\o)$, together with density arguments, gives
\beq\label{p1sc}
\ds\frac{1}{|Y|}\int_{\o\times Y_1} A(y)  (\n u_1+\n_y \widehat{u}_1)\n \f\, dx dy + \frac{1}{|Y|}\int_{\o\times Y_2} A(y) (\n u_1+\n_y \widehat{u}_2) \n \f\, dxdy = \int_\o f\f\, dx, \quad \forall \f \in H^1_0(\o).
\eeq
We now take $w\in\d(\o)$, $\Psi\in H^1_{per}(Y)$ and set $\Psi^\e(x)\doteq\Psi(\frac{x}{\e})$. Then, we use $(\uoe+\e w\Psi^\e, \ute+\e w\Psi^\e)\in K^\e_\g$ as test function in \eqref{pe}, we have
\beq
\ds \int_{\o_1^\e} A^\e\n \uoe\n (\e w\Psi^\e)\, dx  + \int_{\o_2^\e} A^\e\n \ute\n (\e w\Psi^\e)\, dx \geq \e\int_{\o_1^\e} f w\Psi^\e\,dx+\e\int_{\o_2^\e} f w\Psi^\e\,dx.
\eeq
Arguing as in \cite{DoLNTa}, we pass to the limit by unfolding and get
\beq\label{limitesc}
\ds \frac{1}{|Y|}\int_{\o\times Y_1} A(y)(\n u_1+\n_y \widehat{u}_1) \n_y \Phi\, dxdy + \frac{1}{|Y|}\int_{\o\times Y_2} A(y) (\n u_1 +\n_y \widehat{u}_2) \n_y \Phi\, dxdy\geq 0,
\eeq
where $\Phi(x,y)=w(x)\Psi(y)$. By density, \eqref{limitesc} holds for every $\Phi \in L^2(\o;H^1_{per}(Y))$.\\
Taking into account \eqref{u1hat} of Theorem \ref{teconv}, we set
\beq
\ds \widehat{u}\doteq\left\{\ba{ll}
\ds\widehat{u}_1 & \hbox{ in }Y_1,\\[2mm]
\ds \widehat{u}_2+\xi_\G & \hbox{ in }Y_2,
\ea\right.
\eeq
and extend it by periodicity to a function still denoted by $\widehat{u}$, which belongs to $L^2(\o,H^1_{per}(Y))$ and is such that $\mathcal{M}_\G(\widehat{u})=0$ a.e. in $\o$, in view of \eqref{u2hat}. Hence, problem \eqref{limitesc} can be written as follows:
\beq
\ds \frac{1}{|Y|}\int_{\o\times Y} A(y) (\n u_1+\n_y \widehat{u}) \n_y \Phi\, dxdy\geq 0,\quad \forall \Phi \in L^2(\o;H^1_{per}(Y)),
\eeq
which implies
\beq\label{phiuguale}
\ds \frac{1}{|Y|}\int_{\o\times Y} A(y) (\n u_1+\n_y \widehat{u}) \n_y \Phi\, dxdy=0,\quad \forall \Phi \in L^2(\o;H^1_{per}(Y)).
\eeq
By summing up problems \eqref{p1sc} and \eqref{phiuguale}, we obtain
\beq\label{punf2}
\left\{\ba{lll}
\ds \hbox{ Find } (u_1,\widehat{u})\in H^1_0(\o) \times L^2(\o, H^1_{per}(Y)), \hbox{ with }\mathcal{M}_\G(\widehat{u})= 0 \hbox{ a.e. in }\o,\ \hbox{ s.t. }\\[3mm]
\ds \frac{1}{|Y|}\int_{\o\times Y} A(y) (\n u_1+\n_y \widehat{u})(\n \f+ \n_y \Phi)\, dxdy= \int_\o f\f\, dx,\\[5mm]
\ds\forall \f \in H^1_0(\o),\ \forall \Phi \in L^2(\o;H^1_{per}(Y)).
\ea\right.
\eeq
The above problem is the one obtained in the classical work \cite{CiDaGr} dealing with the periodic unfolding method for fixed domains. Hence $(u_1,\widehat{u})$ is unique and convergences \eqref{convseccase} hold for the whole sequences. Moreover, if we define $A^0_\g$ as in \eqref{A0g}, by classical arguments $u_1$ results the unique solution of problem \eqref{probhom}.

\vskip 1cm
\subsection{The case $\g=-1$}
Let $u^\e$ be the weak solution of problem \eqref{prob}. Theorem \ref{teconv}, for $\g=-1$, assures that there exist $u_1\in H^1_0(\o)$, $\widehat{u}_1\in L^2(\o, H^1_{per}(Y_1))$ with $\mathcal{M}_{\G}(\widehat{u}_1)=0$ a.e. in $\o$ and $\widehat{u}_2\in L^2(\o, H^1(Y_2))$ such that convergences \eqref{convseccase} hold, together with
\beq\label{convtircase}
\ds\frac{\toe(\uoe)-\tte(\ute)}{\e}\ru \widehat{u}_1-\widehat{u}_2 \quad \hbox{ weakly in }L^2(\Omega \times \G),
\eeq
up to a subsequence.
We take $\f, w_i\in\d(\o)$, for $i=1,2$, $\Psi_1\in H^1_{per}(Y_1)$ and $\Psi_2\in H^1(Y_2)$ (extended by $Y$-periodicity to $\R^N$) such that
\beq\label{maggiore}
\ds w_1(x)\Psi_1(y) \geq  w_2(x)\Psi_2(y) \hbox{ a.e. in }\o\times\G.
\eeq
Set $\Psi_i^\e(x)\doteq\Psi_i(\frac{x}{\e})$, for $i=1,2$, one has  $w_1(x)\Psi_1^\e(x) \geq w_2(x)\Psi_2^\e(x)$ almost everywhere on $\gae$, so that one can choose $(\f+\e w_1\Psi_1^\e, \f+\e w_2\Psi_2^\e)\in K^\e_\g$ as test function in \eqref{pe} getting
\beq
\ba{lll}
\ds\int_{\o_1^\e} A^\e\n \uoe(\n \f+\n (\e w_1\Psi_1^\e)-\n \uoe) dx+\int_{\o_2^\e} A^\e\n \ute(\n \f+\n (\e w_2\Psi_2^\e)-\n \ute) dx\\[5mm]
\ds +\int_{\gae}h^\e[\ue](w_1\Psi_1^\e-w_2\Psi_2^\e)\,d\sigma\\[5mm]
\ds \geq \int_{\o_1^\e} f(\f+\e w_1\Psi_1^\e-\uoe) dx+\int_{\o_2^\e} f(\f+\e w_2\Psi_2^\e-\ute) dx+\e^{-1}\int_{\gae}h^\e[\ue]^2\,d\sigma.
\ea
\eeq
Let us observe that by \eqref{brd*}
\beq
\e^{-1}\int_{\gae}h^\e[\ue]^2\,d\sigma\geq \frac{\e^{-2}}{|Y|}\int_{\o \times \G} h(y)(\te_1(\ue_1)-\te_2(\ue_2))^2dxd\sigma_y.
\eeq
Hence, arguing analogously to the previous case and taking into account the lower semi-continuity of the norm with respect to the weak convergence, we pass to the limit by unfolding thanks to \eqref{convseccase} and \eqref{convtircase}. Set $\Phi_i(x,y)=w_i(x)\Psi_i(y)$, for $i=1,2$, by density arguments one gets 
\beq
\ba{lll}
%\ds \hbox{ Find } (u_1,\widehat{u}_1,\widehat{u}_2)\in H^1_0(\o) \times L^2(\o, H^1_{per}(Y_1))\times L^2(\o, H^1(Y_2)), \hbox{ with }\mathcal{M}_\G(\widehat{u}_1)= 0 \hbox{ a.e. in }\o,\ \hbox{ s.t. }\\[3mm]
\ds \frac{1}{|Y|}\int_{\o\times Y_1}A(y) (\n u_1+\n_y \widehat{u}_1) (\n (\f-u_1)+\n_y (\Phi_1-\widehat{u}_1))\, dxdy\\[5mm]
\ds + \frac{1}{|Y|}\int_{\o\times Y_2} A(y) (\n u_1 +\n_y \widehat{u}_2) (\n (\f-u_1)+\n_y (\Phi_2-\widehat{u}_2))\, dxdy\\[5mm]
\ds +\frac{1}{|Y|}\int_{\o\times\G}h(y)(\widehat{u}_1-\widehat{u}_2)((\Phi_1-\widehat{u}_1)-(\Phi_2-\widehat{u}_2)\,dxd\sigma_y\geq \int_\o f(\f-u_1)\, dx,\\[5mm]
\ds\forall \f \in H^1_0(\o),\ \forall \Phi_1\in L^2(\o;H^1_{per}(Y_1)),\ \forall \Phi_2 \in L^2(\o;H^1(Y_2)) : \Phi_1 \geq \Phi_2 \hbox{ in }\o\times\G,
\ea
\eeq
in view of \eqref{maggiore}.\\
Let us consider the following spaces:
\beq
\ds \mathcal{B}\doteq\{v=(v_1, \widehat{v}_1,\widehat{v}_2)\in H^1_0(\o)\times L^2(\o,W_{per}(Y_1))\times L^2(\o,H^1(Y_2))\},
\eeq
equipped with the norm
\beq
\ds \|v\|^2_{\mathcal{B}}\doteq \|\n v_1+\n_y \widehat{v}_1\|^2_{L^2(\o\times Y_1)}+\|\n v_1+\n_y \widehat{v}_2\|^2_{L^2(\o\times Y_2)}+\|\widehat{v}_1- \widehat{v}_2\|^2_{L^2(\o\times \G)}, \quad v=(v_1,\widehat{v}_1,\widehat{v}_2)\in \mathcal{B}
\eeq
see \cite[Lemma $5.3$]{donngu}, and
\beq
\ds \mathcal{W}\doteq\{(v_1, \widehat{v}_1,\widehat{v}_2)\in \mathcal{B}\, |\, \widehat{v}_1\geq \widehat{v}_2 \hbox{ in }\o\times \G\}.
\eeq
%We note that $\overline{W}$, given in \eqref{Wbar}, is a closed convex subset of $W_{per}(Y_1)\times H^1(Y_2)$ and we define
%\beq
%\ds L^2(\o, \overline{W})\doteq \{(v_1,v_2)\in L^2(\o, W_{per}(Y_1))\times L^2(\o,H^1(Y_2))|(v_1,v_2)(x,\cdot)\in \overline{W} \hbox{ a.e. }x\in \o \}.
%\eeq
Now, let us observe that
\beq\label{ucappucci}
\ds\widehat{u}_1\geq \widehat{u}_2\quad \hbox{ a.e. in }\o\times\G.
\eeq
Indeed, since $\uoe\geq\ute$ a.e. on $\gae$, one deduces that $\te_1(\uoe)\geq\te_2(\ute)$ a.e. in $\o\times \G$, hence \eqref{ucappucci} follows from convergence \eqref{convtircase}.\\
Thus, $(u_1,\widehat{u}_1,\widehat{u}_2)$ satisfies
\beq\label{puftc}
\left\{\ba{lll}
\ds \hbox{ Find } (u_1,\widehat{u}_1,\widehat{u}_2)\in \mathcal{W}\ \hbox{ s.t. }\\[3mm]
\ds \frac{1}{|Y|}\int_{\o\times Y_1}A(y) (\n u_1+\n_y \widehat{u}_1) (\n (\f-u_1)+\n_y (\Phi_1-\widehat{u}_1))\, dxdy\\[5mm]
\ds + \frac{1}{|Y|}\int_{\o\times Y_2} A(y) (\n u_1 +\n_y \widehat{u}_2) (\n (\f-u_1)+\n_y (\Phi_2-\widehat{u}_2))\, dxdy\\[5mm]
\ds +\frac{1}{|Y|}\int_{\o\times\G}h(y)(\widehat{u}_1-\widehat{u}_2)((\Phi_1-\widehat{u}_1)-(\Phi_2-\widehat{u}_2)\,dxd\sigma_y\geq \int_\o f(\f-u_1)\, dx,\\[5mm]
\ds\forall \f \in H^1_0(\o),\ \forall \Phi_1\in L^2(\o;H^1_{per}(Y_1)),\ \forall \Phi_2 \in L^2(\o;H^1(Y_2)) : \Phi_1 \geq \Phi_2 \hbox{ in }\o\times\G.
\ea\right.
\eeq
Taking into account our assumptions, by standard arguments on variational inequalities, one also gets that problem \eqref{puftc} admits a unique solution. Therefore convergences \eqref{convseccase} and \eqref{convtircase} hold for the whole sequences.

In order to describe the homogenized problem satisfied by $u_1$, we replace the test function $\f$ by $\f+u_1$ and, in view of \eqref{ucappucci}, $\Phi_i$ by $\widehat{u}_i$, for $i=1,2$, in \eqref{puftc}:
\beq
\ba{llll}
\ds \frac{1}{|Y|}\int_{\o\times Y_1}A(y) (\n u_1+\n_y \widehat{u}_1) \n \f\, dxdy + \frac{1}{|Y|}\int_{\o\times Y_2} A(y) (\n u_1 +\n_y \widehat{u}_2) \n \f \, dxdy\geq \int_\o f\f\, dx,\quad \forall \f \in H^1_0(\o).
\ea
\eeq
Now, taking $\f=-\f$ as test function we obtain
\beq\label{p1tc} 
\ds\frac{1}{|Y|}\int_{\o\times Y_1} A(y) (\n u_1+\n_y \widehat{u}_1)\n \f\, dx dy + \frac{1}{|Y|}\int_{\o\times Y_2} A(y)(\n u_1+\n_y \widehat{u}_2) \n \f\, dxdy = \int_\o f\f\, dx, \quad \forall \f \in H^1_0(\o).
\eeq
Lastly, by taking $\f=u_1$ as test function in \eqref{puftc}, one gets
\beq\label{carmen}
\ba{ll}
\ds \frac{1}{|Y|}\int_{\o\times Y_1}A(y) (\n u_1+\n_y \widehat{u}_1) (\n_y \Phi_1-\n_y \widehat{u}_1)\, dxdy + \frac{1}{|Y|}\int_{\o\times Y_2} A(y) (\n u_1 +\n_y \widehat{u}_2) (\n_y \Phi_2-\n_y \widehat{u}_2)\, dxdy\\[5mm]
\ds +\frac{1}{|Y|}\int_{\o\times\G}h(y)(\widehat{u}_1-\widehat{u}_2)((\Phi_1-\widehat{u}_1)-(\Phi_2-\widehat{u}_2))\,dxd\sigma_y\geq 0,\\[5mm]
\ds \forall \Phi_1\in L^2(\o;H^1_{per}(Y_1)),\ \forall \Phi_2 \in L^2(\o;H^1(Y_2)) : \Phi_1 \geq \Phi_2 \hbox{ in }\o\times\G.
\ea 
\eeq
To conclude, observe that classical arguments on variational inequalities give that the solution of the cell problem \eqref{celltircase} exists and it is unique, by endowing the space $H^1_{per}(Y_1)\times H^1(Y_2)$ with the equivalent norm
\beq
\ds \|(z_1,z_2)\|^2_{H^1_{per}(Y_1)\times H^1(Y_2)}\doteq \|\n z_1\|^2_{L^2(Y_1)}+\|\n z_2\|^2_{L^2(Y_2)}+\|z_1-z_2\|^2_{L^2(\G)}.
\eeq
Then, by the uniqueness of the solution of \eqref{celltircase} written for $\z=\n u_1$, we obtain that $\widehat{u}_1$ and $\widehat{u}_2$ in  \eqref{carmen}  can be expressed  as 
\beq\label{sara}
\widehat{u}_1(x,y)=\widehat{\chi}_1(y,\n u_1(x)) \quad \hbox{ and }\quad \widehat{u}_2(x,y)=\widehat{\chi}_2(y,\n u_1(x)).
\eeq
Finally, set $A^0_{-1}$ as in \eqref{ahom}, by \eqref{p1tc} and \eqref{sara} the homogenized problem satisfied by $u_1$ can be rewritten in the following form:
\beq
\ds\int_{\o}A^0_{-1}(\n u_1)\n \f\, dx = \int_\o f\f\, dx, \quad \forall \f \in H^1_0(\o),
\eeq
which gives problem \eqref{probhom}. The uniqueness of solutions of problems \eqref{celltircase} and \eqref{puftc} easily implies the uniqueness of problem \eqref{probhom}.

\subsection{The case $\g\in ]-1;1[$}
Let $u^\e$ be the weak solution of problem \eqref{prob}. Theorem \ref{teconv}, for $\g\in ]-1;1[$, assures that there exist $u_1\in H^1_0(\o)$, $\widehat{u}_1\in L^2(\o, H^1_{per}(Y_1))$ with $\mathcal{M}_{\G}(\widehat{u}_1)=0$ a.e. in $\o$ and $\overline{u}_2\in L^2(\o, H^1(Y_2))$ with $\mathcal{M}_{\G}(\overline{u}_2)=0$ a.e. in $\o$ such that convergences \eqref{convseccase} hold, up to a subsequence, where the last convergence is equivalent to
\bk\beq\label{convfirstcase}
\ds \te_2(\n\ue_2)\ru \n_y \overline{u}_2 \quad \hbox{ weakly in }L^2(\o\times Y_2),
\eeq
in view of \eqref{u2hat}.

We want to prove that $\n_y \overline{u}_2=0$ a.e. in $\o\times Y_2$. To this aim, we take $w\in\d(\o)$, $\Psi_2\in C^1_c(\overline{Y}_2)$, set $\Psi_2^\e(x)\doteq\Psi_2(\frac{x}{\e})$ and extend it by periodicity. From the boundedness of $w$ and $\Psi_2$, there exists $M\geq 0$ such that
\beq\label{em}
\ds \|w\|_{L^{\infty}(\o)}\|\Psi_2\|_{L^{\infty}(\overline{Y}_2)}\leq M.
\eeq
Hence, $\uoe-\ute+\e w\Psi_2^\e+\e M\geq 0$ a.e. on $\gae$, so that $(\uoe, \ute-\e w\Psi_2^\e-\e M)\in K^\e_\g$. Therefore, we can use it as test function in \eqref{pe} getting
\beq\label{eps}
\ds -\int_{\o_2^\e} A^\e\n \ute\n (\e w\Psi_2^\e)\, dx +\e^{\g+1}\int_{\gae}h^\e[\ue](w\Psi_2^\e+M)\,d\sigma\geq -\e\int_{\o_2^\e} f (w\Psi_2^\e+M)\,dx.
\eeq

By Lemma \ref{phibordo} and \eqref{teconv}, one has
\beq\label{B1}
\ba{ll}
\ds \e^{\g+1}\int_{\gae}h^\e[\ue]w\Psi_2^\e\,d\sigma=\e^\g\frac{1}{|Y|}\int_{\o \times\G}h(y)(\te_1(\uoe)-\te_2(\ute))\te_2(w)\Psi_2(y)dxd\sigma_y\\[5mm]
\ds \leq C\e^\g \|\te_1(\uoe)-\te_2(\ute)\|_{L^2(\o \times\G)}\|\te_2(w)\Psi_2(y)\|_{L^2(\o \times\G)}\leq C\e^\g\e^{\frac{1-\g}{2}}=C\e^{\frac{1+\g}{2}}.
\ea
\eeq
On the other hand
\beq\label{B2}
\ba{ll}
\ds \e^{\g+1}M\int_{\gae}h^\e[\ue]\,d\sigma \leq C\e^{\g+1} \|[\ue]\|_{L^2(\gae)}|\gae|^{\frac{1}{2}}\leq C\e^{\frac{1+\g}{2}}.
\ea
\eeq
Since $\g>-1$, by passing to the limit in \eqref{B1} and \eqref{B2}, as $\e$ tends to zero, we obtain
\beq\label{B3}
\ds \lim_{\e\rw 0}\e^{\g+1}\int_{\gae}h^\e[\ue](w\Psi_2^\e+M)\,d\sigma=0
\eeq
and, thus, \eqref{eps} gives
\beq\label{limite}
\ds \frac{1}{|Y|}\int_{\o\times Y_2}A\n_y \overline{u}_2 \n_y \Phi_2\, dxdy\leq 0,
\eeq
where $\Phi_2(x,y)=w(x)\Psi_2(y)$. By density, \eqref{limite} holds for every $\Phi_2 \in L^2(\o;H^1(Y_2))$.\\
Now, by taking $\Phi_2=\overline{u}_2$ in \eqref{limite} and using (\textbf{A$1$}), one obtains that 
\beq\label{grad0}
\ds\n_y \overline{u}_2=0\quad \hbox{ a.e. in } \o\times Y_2.
\eeq
Arguing as for the case $\g<-1$, we take $\f\in\d(\o)$ and use $(\uoe +\f, \ute+\f)\in K^\e_\g$ as test function in \eqref{pe} obtaining
\beq\label{p1firstcase}
\ds\frac{1}{|Y|}\int_{\o\times Y_1} A (\n u_1+\n_y \widehat{u}_1)\n \f\, dx dy = \int_\o f\f\, dx, \quad \forall \f \in H^1_0(\o),
\eeq
in view of \eqref{grad0}.\\
Let $C^\infty_{per}(Y_1\cup\Gamma)\doteq\{v_{|Y_1\cup\Gamma}\,|\,v\in C^\infty_{per}(Y)\}$. It is easily seen that $\overline{C^\infty_{per}(Y_1\cup\Gamma)}= H^1_{per}(Y_1)$.
Hence, if we take $w\in\d(\o)$ and $\Psi_1\in C^\infty_{per}(Y_1\cup\Gamma)$, there exists $M\geq 0$ such that
\beq\label{em}
\ds \|w\|_{L^{\infty}(\o)}\|\Psi_1\|_{L^{\infty}(Y_1\cup \Gamma)}\leq M.
\eeq
Thus, set $\Psi_1^\e(x)\doteq\Psi_1(\frac{x}{\e})$, we can choose $(\uoe+\e w\Psi_1^\e+\e M, \ute)\in K^\e_\g$ as test function in \eqref{pe} getting:
\beq
\ds \int_{\o_1^\e} A^\e\n \uoe\n (\e w\Psi_1^\e)\, dx +\e^{\g+1}\int_{\gae}h^\e[\ue](w\Psi_1^\e+M)\,d\sigma\geq \e\int_{\o_1^\e} f (w\Psi_1^\e+M)\,dx.
\eeq
By unfolding and arguing as in \eqref{B2} and \eqref{B3}, one has
\beq
\ds \frac{1}{|Y|}\int_{\o\times Y_1}A(y)(\n u_1+\n_y \widehat{u}_1) \n_y \Phi_1\, dxdy\geq 0,
\eeq
where $\Phi_1(x,y)=w(x)\Psi_1(y)$. By density, the previous inequality holds for every $\Phi_1 \in L^2(\o;H^1_{per}(Y_1))$, which implies
\beq\label{p2firstcase}
\ds \frac{1}{|Y|}\int_{\o\times Y_1}A(y)(\n u_1+\n_y \widehat{u}_1) \n_y \Phi_1\, dxdy= 0,\quad \forall \Phi_1\in L^2(\o;H^1_{per}(Y_1)).
\eeq
Finally, we add \eqref{p1firstcase} and \eqref{p2firstcase}, and obtain the following problem:
\beq\label{prunffirtscase}
\left\{\ba{lll}
\ds \hbox{ Find } (u_1,\widehat{u}_1)\in H^1_0(\o) \times L^2(\o, H^1_{per}(Y_1)), \hbox{ with }\mathcal{M}_\G(\widehat{u}_1)= 0 \hbox{ a.e. in }\o,\ \hbox{ s.t. }\\[5mm]
\ds \frac{1}{|Y|}\int_{\o\times Y_1} A(y) (\n u_1+\n_y \widehat{u}_1)(\n \f+\n_y \Phi_1)\, dx dy =\int_\o f\f\, dx,\\[5mm]
\ds \forall \f \in H^1_0(\o), \quad \forall \Phi_1\in L^2(\o;H^1_{per}(Y_1)).
\ea\right.
\eeq
The above problem is the one obtained in the classical work \cite{up5} dealing with the periodic unfolding method for perforated domains. Hence $(u_1,\widehat{u}_1)$ is unique and convergences \eqref{convseccase} and \eqref{convfirstcase} (where $\n_y \overline{u}_2\equiv 0$) hold for the whole sequences. 
Moreover, if we define $A^0_\g$ as in \eqref{a0g}, by classical arguments $u_1$ results the unique solution of problem \eqref{probhom}.

\vskip 1cm
\subsection{The case $\g=1$}
Let $u^\e$ be the weak solution of problem \eqref{prob}. Arguing as in the case $\gamma\in ]-1;1[$, by Theorem \ref{teconv} we get again \eqref{grad0}, even in the case $\g=1$. We deduce also the existence of $u_1\in H^1_0(\o)$, $\widehat{u}_1\in L^2(\o, H^1_{per}(Y_1))$ with $\mathcal{M}_{\G}(\widehat{u}_1)=0$ a.e. in $\o$ and  ${u}_2 \in L^2(\o)$ such that, up to a subsequence,
\beq\label{convlastcase}
\left\{\ba{lllll}
\ds \widetilde{u^\e_i} \ru \theta_i u_i & \hbox{ weakly in } L^2(\o),\ i=1,2, \\[3mm]
\ds\te_1(\ue_1)\rw u_1 & \hbox{ strongly in } L^2(\o, H^1(Y_1)),\\[3mm]
\ds\te_2(\ue_2)\ru u_2 & \hbox{ weakly in }L^2(\o, H^1(Y_2)),\\[3mm]
\ds\te_1(\n\ue_1)\ru \n u_1 +\n_y \widehat{u}_1 & \hbox{ weakly in }L^2(\o\times Y_1),\\[3mm]
\ds \te_2(\n\ue_2)\ru 0 & \hbox{ weakly in }L^2(\o\times Y_2),
\ea\right.
\eeq
and
\beq\label{phiucappuccio}
\ds \frac{1}{|Y|}\int_{\o\times Y_1}A(y)(\n u_1+\n_y \widehat{u}_1) \n_y \Phi_1\, dxdy= 0,\quad \forall \Phi_1\in L^2(\o;H^1_{per}(Y_1)).
\eeq
By standard arguments, it follows that
\beq\label{ucappuccio1}
\widehat{u}_1(x,y)=-\sum_{j=1}^N \frac{\partial u_1}{\partial x_j}(x) \chi_j(y)
\eeq
where $\chi_j$, for $j=1, ..., N$, are the unique solutions of the cell problems \eqref{cell1c}.

Now, we take $\f_i\in\d(\o)$, $i=1,2$, such that $\f_1\geq\f_2$ in $\o$. By choosing $(\f_1,\f_2)\in K^\e_\g$ as test function in \eqref{pe} we get
\beq
\ba{ll}
\ds\int_{\o_1^\e} A^\e\n \uoe\n (\f_1-\uoe)dx +\int_{\o_2^\e} A^\e\n \ute\n (\f_2-\ute)dx +\e\int_{\gae}h^\e[\ue][\f]d\sigma\\[3mm]
\ds\geq \int_{\o_1^\e} f(\f_1-\uoe)dx+\int_{\o_2^\e} f(\f_2-\ute)dx+\e\int_{\gae}h^\e([\ue])^2d\sigma.
\ea
\eeq
We want to pass to the limit by unfolding in the previous inequality. To this aim let us observe that by \eqref{brd*}
\beq\label{B5}
\e\int_{\gae}h^\e[\ue]^2\,d\sigma\geq \frac{1}{|Y|}\int_{\o \times \G} h(y)(\te_1(\ue_1)-\te_2(\ue_2))^2dxd\sigma_y.
\eeq
Hence, by \eqref{B5} and the lower semi-continuity of the norm with respect to the weak convergence, we obtain
\beq\label{prima}
\ba{lll}
\ds \frac{1}{|Y|}\int_{\o\times Y_1} A(y) (\n u_1+\n_y \widehat{u}_1)(\n \f_1-\n u_1-\n_y\widehat{u}_1)\, dx dy+\frac{1}{|Y|}\int_{\o\times \G} h(y) (u_1-u_2)(\f_1-\f_2)\, dx dy\\[5mm]
\ds \geq \theta_1\int_\o f(\f_1-u_1)\, dx+\theta_2\int_\o f(\f_2-u_2)\, dx+\frac{1}{|Y|}\int_{\o\times \G} h(y) (u_1-u_2)^2\, dx dy.
\ea
\eeq
If we write \eqref{phiucappuccio} for $\Phi_1=\widehat{u}_1$ and we add it to \eqref{prima}, we get
\beq
\ba{lll}
\ds \frac{1}{|Y|}\int_{\o\times Y_1} A(y) (\n u_1+\n_y \widehat{u}_1)(\n \f_1-\n u_1)\, dx dy+\frac{1}{|Y|}\int_{\o\times \G} h(y) (u_1-u_2)((\f_1-\f_2)-(u_1-u_2))\, dx dy\\[5mm]
\ds \geq \theta_1\int_\o f(\f_1-u_1)\, dx+\theta_2\int_\o f(\f_2-u_2)\, dx.
\ea
\eeq

We observe that
\beq\label{u12}
\ds u_1\geq u_2\quad \hbox{ a.e. in }\o.
\eeq
Indeed, since $\uoe\geq\ute$ a.e. on $\gae$, one deduces that $\te_1(\uoe)\geq\te_2(\ute)$ a.e. in $\o\times \G$, hence \eqref{u12} follows from convergences \eqref{convlastcase}.

Hence, in view of \eqref{ucappuccio1}, \eqref{u12} and the fact that $A^0_1$ is given by \eqref{a0g}, set $K\doteq\{(\f_1,\f_2)\in H^1_0(\o)\times L^2(\o) \,|\, \f_1\geq \f_2 \hbox{ in }\o\}$, we get that $(u_1,u_2)$ is solution of the following problem:
\beq
\left\{\ba{lll}
\ds \hbox{ Find } (u_1,u_2)\in K \hbox{ s.t. }\\[3mm]
\ds \int_{\o} A^0_1 \n u_1(\n \f_1-\n u_1)\, dx +\frac{\mathcal{M}_\G(h)}{|Y|}\int_{\o} (u_1-u_2)[(\f_1-\f_2)-(u_1-u_2)]\, dx \\[5mm]
\ds \geq \theta_1\int_\o f(\f_1-u_1)\, dx+\theta_2\int_\o f(\f_2-u_2)\, dx,\quad \forall (\f_1,\f_2)\in K.
\ea\right.
\eeq
Let us denote by $\mathcal{A}$ the operator
$$ \mathcal{A}: K \rw K'$$ given by
\beq
\left\langle\mathcal{A}(u),v\right\rangle_{K',K}\doteq\int_{\o} A^0_1 \n u_1\n v_1\, dx +\frac{\mathcal{M}_\G(h)}{|Y|}\int_{\o} (u_1-u_2)(v_1-v_2)dx.
\eeq
By $(\textbf{A}4)$ and the fact that $A^0_1\in \mathcal{M}(\alpha, \frac{\beta^2}{\alpha}, \o)$ (see \cite{CSJP}), one has
\beq
\ba{ll}
\ds \left\langle\mathcal{A}(w)-\mathcal{A}(v),w-v\right\rangle=\int_{\o} A^0_1 \n (w_1-v_1)\n (w_1-v_1)\, dx +\frac{\mathcal{M}_\G(h)}{|Y|}\int_{\o} [(w_1-w_2)-(v_1-v_2)]^2\, dx\\[5mm]
\ds \geq \alpha \|w_1-v_1\|_{H^1_0(\o)}^2+h_0 \|(w_1-w_2)-(v_1-v_2)\|_{L^2(\o)}^2>0, \quad \hbox{ if }w\neq v.
\ea
\eeq
From this and \cite[Thereom 8.3 (pag. $248$)]{lions}, we deduce that the couple $(u_1,u_2)$ is unique. Hence, by \eqref{ucappuccio1}, convergences \eqref{convlastcase} hold for the whole sequences.

%
%\blu se la mettiamo, bisogna definire $\te$:
%\begin{remark}
%...Instead of $(A1)$, one can suppose that there exists a matrix $B$ such that
%$$B^\e=\te(A^\e) \rw B\quad \hbox{ a.e. (or in measure) in }\Omega \times Y.$$
%This implies that $B\in\mathcal{M}(\alpha, \beta, \Omega \times Y)$...
%\end{remark}\bk

\section*{Acknowledgement}
This work was supported by the project PRIN2022 D53D23005580006 ``Elliptic and parabolic problems, heat kernel estimates and spectral theory" and by the projects GNAMPA2023 (CUP E53C22001930001) ``Problemi  misti locali e non locali con coefficienti singolari" and ``Analisi asintotica di problemi al contorno in strutture sottili o rugose".


\begin{thebibliography}{99}
				
\bibitem{AmAnT} M. Amar, D. Andreucci and C. Timofte, Heat conduction in composite media involving imperfect contact and perfectly conductive inclusions, {\it Math. Methods Appl. Sci.} {\bf 45} (2022) 11355--11379.


\bibitem{AMR} J. Avila, S. Monsurr\`{o} and F. Raimondi, Homogenization of an eigenvalue problem through rough surfaces, {\it Asymptot. Anal.} {\bf 137} (2024) no. 1-2, 97--121.

\bibitem{ben}
A. Bensoussan, J.-L. Lions and G. Papanicolaou, {\it Asymptotic analysis for periodic structures}, (North-Holland Publishing Company, Netherlands, 1978).

\bibitem{Bre}
H. Brezis,  Problemes unilateraux, {\it J. Math. Pures Appl.} {\bf 51} (1972) 1--168. 

\bibitem{RCk}
R. Bunoiu, K. Ramdani and C. Timofte, Homogenization of a transmission problem with sign-changing coefficients and interfacial flux jump, {\it  Commun. Math. Sci.} {\bf 21} (2023) 2029--2049.

\bibitem{RT}
R. Bunoiu and C. Timofte, Homogenization of a thermal problem with flux jump, {\it Networks and Heterogeneous Media} {\bf 11} (2016) no.4, 17


\bibitem{CET}
A. Capatina, H. Ene and C. Timofte, Homogenization results for elliptic problems in periodically perforated domains with mixed-type boundary conditions, {\it Asymptotic Analysis} (2012) 1--12.
		
%\rd \bibitem{CPT} G. Cardone, C. Perugia and C. Timofte, Homogenization results for a coupled system of reaction diffusion equations, {\it Nonlinear Anal.} {\bf 188} (2019) 236--264.\bk
		  
\bibitem {up5}
D. Cioranescu, A. Damlamian, P. Donato, G. Griso and R. Zaki, The periodic unfolding method in domains with holes, {\it SIAM J. Math. Anal.} {\bf 44} (2012) 718--760.

\bibitem {CiDaGr}
D. Cioranescu, A. Damlamian and G. Griso, Periodic unfolding and homogenization, {\it C. R. Math. Acad. Sci. Paris Ser. I} {\bf 335} (2002) 99--104. 
       
\bibitem{DGO}
D.Cioranescu, A. Damlamian and J. Orlik, Homogenization via unfolding in periodic elasticity with contact on closed and open cracks, {\it  Asymptotic Analysis} {\bf 82} (2013) 201--232.

\bibitem{d}
D. Cioranescu, P. Donato and R. Zaki, The periodic unfolding method in perforated domains, {\it Port Math.} {\bf 63} (2006) 476--496. 

\bibitem{cp}
D. Cioranescu and J. Saint Jean Paulin, Homogenization in open sets with holes, { \it J. Math. Anal. Appl.} {\bf 71} (1979) 590--607. 

\bibitem{CSJP} 
D. Cioranescu and J. Saint Jean Paulin, {\it Homogenization of Reticulated Structures}, (Springer-Verlag, New York, 1999).

\bibitem{CMT}
C. Conca, F. Murat and C. Timofte, A Generalized strange term in Singorini's type problems, {\it ESAIM: M2AN} {\bf 37} (2003) 773--805.


\bibitem{DoJo} 
P. Donato and E.C. Jose, Corrector results for a parabolic problem with a memory effect, {\it ESAIM: M2AN} {\bf 44} (2010) 421--454.
			
	 
\bibitem{donngu}
P. Donato and K.H. Le Nguyen, Homogenization of diffusion problems with a nonlinear interfacial resistance, {\it Nonlinear Diff. Equ. Appl.} {\bf 22} (2015) 1345--1380.
			
\bibitem{DoLNTa}
P. Donato, K.H. Le Nguyen,  and R. Tardieu, The periodic unfolding method for a class of imperfect transmission problems, {\it J. Math. Sci.} {\bf 176} (2011) 891--927.

\bibitem{DoMo}
P. Donato, and S. Monsurrò, Homogenization of two heat conductors with an interfacial contact resistance, {\it Anal. Appl.} {\bf 3} (2004) 247--273. 

\bibitem{DL}
G. Duvaut and J.L. Lions, {\it Les in\'equations en m\'ecanique et en physique} (Dunod, Paris, 1972).

\bibitem{FMPhomo} 
L. Faella, S. Monsurr\`o and C. Perugia, Homogenization of imperfect transmission problems: the case of weakly converging data, {\it Differ. Integral Equ.} {\bf 31} (2018) 595--620.
		
\bibitem{FMP} 
L. Faella, S. Monsurr\'o and C. Perugia, Exact controllability for evolutionary imperfect transmission problems, {\it J. Math. Pures Appl.} {\bf 122} (2019) 235--271.

\bibitem{fichera}
G. Fichera, Problemi elastostatici con vincoli unilaterali: il problema di Signorini con ambigue condizioni al contorno,  {\it Mem. Accad. Naz. Lincei} {\bf 87} (1964) 91--140.
(1964)

\bibitem{GM1}
A. Gaudiello and T. Mel’nyk, Homogenization of a nonlinear monotone problem with nonlinear Signorini boundary conditions in a domain with highly rough boundary, {\it J. Differential Equations} {\bf 265} (2018) 5419--5454.

\bibitem{GM2}
A. Gaudiello and T. Mel’nyk, Homogenization of a nonlinear monotone problem with a big nonlinear Signorini boundary interaction
in a domain with highly rough boundary, {\it Nonlinearity} {\bf 32} (2019) 5150--5169.

\bibitem{GMO}
G. Griso, A. Migunova and J. Orlik, Homogenization via unfolding in periodic layer with contact, {\it Asymptotic Analysis} {\bf 99} (2016) 23--52. 

\bibitem{hum} 
H.C. Hummel, Homogenization for heat transfer in polycristals with interfacial resistances, {\it Appl. Anal.} {\bf 75} (2000) 403--424.

\bibitem{KS}
D. Kinderleherer and G. Stampacchia, {\it An Introduction to Variational Inequalities and Their Applications, Classics in Applied Mathematics} (Academic Press, New York, 2000). 

\bibitem{lions}
J.L. Lions, {\it Quelques m\'ethodes de r\'esolution des probl\`{e}mes aux limites non lin\'eaires}, (Dunod, Paris, 1969).


\bibitem{LS}
J.L. Lions and G. Stampacchia, Variational inequalities, {\it Commun. Pure Appl. Math.} {\bf 20} (1967) 493--519. 

\bibitem{MNW}
T.A. Mel’nyk, Iu.A. Nakvasiuk and W.L. Wendland, Homogenization of the Signorini boundary-value problem in a thick junction and boundary integral equations for the homogenized problem, {\it Math. Meth. Appl. Sci.}{\bf 34} (2011) 758--775.

\bibitem{Mo}
S. Monsurr\`o, Homogenization of a two-component composite with interfacial thermal barrier, {\it Adv. Math. Sci. Appl.} {\bf 13} (2003) 43--63.

\bibitem{Moer}
S. Monsurr\`o, Erratum for the paper homogenization of a two-component composite with interfacial thermal barrier, {\it Adv. Math. Sci. Appl.} {\bf 14} (2004) 375--377. 

\bibitem{MNP}
S. Monsurrò, A.K. Nandakumaran and C. Perugia, Exact internal controllability for a problem with imperfect interface, {\it Appl. Math. Optim.} {\bf 85} (2022) no. 3, Paper No. 27, 33 pp.

\bibitem{MNP2}
S. Monsurr\`{o}, A.K. Nandakumaran and C. Perugia, A Note on the Exact Boundary Controllability for an Imperfect Transmission Problem,  {\it Ric. Mat.} {\bf 73} (2024) no. 1, 547--564.

\bibitem{MP1} 
S. Monsurr\`{o} and C. Perugia, Homogenization and exact controllability for problems with imperfect interface, {\it Netw. Heterog. Media} {\bf 14} (2019) 411--444.

\bibitem{mpr} 
S. Monsurr\`{o}, C. Perugia and F. Raimondi, Homogenization of a nonlinear elliptic problem with imperfect rough interface, {\it  submitted}.

\bibitem{pastu}
S.E. Pastukhova, Homogenization of a mixed problem with Signorini condition for an elliptic operator in a perforated domain, {\it Sb. Mat.} {\bf 192} (2001) no. 2, 245 pp.

\bibitem{Rai} 
F. Raimondi, Homogenization of a class of singular elliptic problems in two-component domains, {\it Asymptot. Anal.} {\bf 132} (2023) 1--27.

\bibitem{Signo}
A. Signorini, Questioni di elasticita non linearizzata o semilinearizzata, {\it Rendiconti di Matematica e delle sue Applicazioni} {\bf 18} (1959) 95--139.

\bibitem{Signo1}
A. Signorini, Sopra alcune questioni di Elastostatica, {\it Atti della Soc. Ital. per il Progresso della Scienze} (1963).
\end{thebibliography}
\end{document}